\documentclass[12pt,leqno]{amsart}
\usepackage{amssymb,mathrsfs}
\begin{document}
\theoremstyle{plain}
\newtheorem{thm}{Theorem}[section]
\newtheorem{prop}[thm]{Proposition}
\newtheorem{lemma}[thm]{Lemma}
\newtheorem{clry}[thm]{Corollary}
\newtheorem*{definition}{Definition}
\newtheorem{hyp}{Assumption}
\newtheorem{claim}{Claim}
\newtheorem*{CSHyp}{Carleson-Sj\"olin Assumption}

\theoremstyle{definition}
\newtheorem{remark}[thm]{Remark}
\numberwithin{equation}{section}
\newcommand{\eps}{\varepsilon}
\newcommand{\e}{\mathrm{e}}
\renewcommand{\phi}{\varphi}
\renewcommand{\d}{\partial}
\newcommand{\dd}{\mathrm{d}}
\newcommand{\re}{\mathop{\rm Re} }
\newcommand{\im}{\mathop{\rm Im}}
\newcommand{\R}{\mathbf{R}}
\newcommand{\T}{\mathbf{T}}
\renewcommand{\S}{\mathbf{S}}
\newcommand{\C}{\mathbf{C}}
\newcommand{\N}{\mathbf{N}} 
\newcommand{\Z}{\mathbf{Z}} 
\newcommand{\D}{C^{\infty}_0} 
\newcommand{\supp}{\mathop{\rm supp}}

\newcommand{\mBic}{\,\Leftrightarrow\,}         
\newcommand{\diam}{\mathrm{diam}}
\newcommand{\diag}{\mathrm{diag}}
\newcommand{\mker}{\mathrm{ker}}                
\newcommand{\mim}{\mathrm{im}}                  
\newcommand{\mdet}{\mathrm{det}}                
\newcommand{\tr}{\mathrm{tr}}
\newcommand{\adj}{\mathrm{adj}}
\newcommand{\mR}{\mathbf{R}}                    
\newcommand{\mC}{\mathbf{C}}                    
\newcommand{\mK}{\mathbf{K}}                    
\newcommand{\mZ}{\mathbf{Z}}                    
\newcommand{\mN}{\mathbf{N}}                    
\newcommand{\abs}[1]{\lvert #1 \rvert}          
\newcommand{\Bigabs}[1]{\Big\lvert #1 \Big\rvert}               
\newcommand{\norm}[1]{\lVert #1 \rVert}         
\newcommand{\ip}[2]{\langle #1, #2 \rangle}     
\newcommand{\br}[1]{\langle #1 \rangle}         
\newcommand{\mRS}{R_{S}}                        
\newcommand{\mCS}{C_{S}}                        

\newcommand{\tphi}{\tilde{\varphi}}
\newcommand{\hphi}{\hat{\varphi}}

\newcommand{\mD}{\mathscr{D}}
\newcommand{\mDn}{\mathscr{D}(\mR^n)}
\newcommand{\mDp}{\mathscr{D}^{\prime}}
\newcommand{\mDpn}{\mathscr{D}^{\prime}(\mR^n)}
\newcommand{\mE}{\mathscr{E}}
\newcommand{\mEn}{\mathscr{E}(\mR^n)}
\newcommand{\mEp}{\mathscr{E}^{\prime}}
\newcommand{\mEpn}{\mathscr{E}^{\prime}(\mR^n)}
\newcommand{\mS}{\mathscr{S}}
\newcommand{\mSn}{\mathscr{S}(\mR^n)}
\newcommand{\mSp}{\mathscr{S}^{\prime}}
\newcommand{\mSpn}{\mathscr{S}^{\prime}(\mR^n)}
\newcommand{\mB}{\mathscr{B}}
\newcommand{\mBn}{\mathscr{B}(\mR^n)}
\newcommand{\mBp}{\mathscr{B}^{\prime}}
\newcommand{\mBpn}{\mathscr{B}^{\prime}(\mR^n)}
\newcommand{\mOm}{\mathscr{O}_M}
\newcommand{\mOmn}{\mathscr{O}_M(\mR^n)}
\newcommand{\mOcp}{\mathscr{O}_C^{\prime}}
\newcommand{\mOcpn}{\mathscr{O}_C^{\prime}(\mR^n)}

\newcommand{\mU}{\mathscr{U}}
\newcommand{\mUn}{\mathscr{U}(\mR^n)}

\newcommand{\mP}{\mathscr{P}}
\newcommand{\mPn}{\mathscr{P}(\mR^n)}
\newcommand{\mPp}{\mathscr{P}^{\prime}}
\newcommand{\mPpn}{\mathscr{P}^{\prime}(\mR^n)}

\newcommand{\ms}{\mathscr{S}(\mZ^n)}
\newcommand{\msn}{\mathscr{S}(\mZ^n)}
\newcommand{\msp}{\mathscr{S}^{\prime}(\mZ^n)}
\newcommand{\mspn}{\mathscr{S}^{\prime}(\mZ^n)}
\newcommand{\mo}{\mathscr{O}(\mZ^n)}
\newcommand{\mon}{\mathscr{O}(\mZ^n)}

\newcommand{\mF}{\mathscr{F}}

\newcommand{\ehat}{\,\hat{\rule{0pt}{6pt}}\,}
\newcommand{\echeck}{\,\check{\rule{0pt}{6pt}}\,}
\newcommand{\etilde}{\,\tilde{\rule{0pt}{6pt}}\,}

\newcommand{\id}{\mathrm{Id}}
\newcommand{\cl}{\mathrm{cl}}

\newcommand{\nli}[1]{\lVert #1 \rVert_{L^{\infty}}}             
\newcommand{\xnli}[1]{\lVert \br{x} #1 \rVert_{L^{\infty}}}
\newcommand{\nlt}[1]{\lVert #1 \rVert_{L^{2}}}                  
\newcommand{\nltd}[1]{\lVert #1 \rVert_{L^2_{\delta}}}          
\newcommand{\nltdp}[1]{\lVert #1 \rVert_{L^2_{\delta+1}}}       

\newcommand{\mH}{\mathscr{H}}

\newcommand{\mdiv}{\mathrm{div}}
\newcommand{\mOp}{\mathrm{Op}}
\newcommand{\proj}{\mathrm{proj}}

\title[]{Determining an unbounded potential \\ from Cauchy data in admissible geometries}
\author[]{David Dos Santos Ferreira \and Carlos E.~Kenig \and Mikko Salo}
\address{Universit\'e Paris 13, Cnrs, Umr 7539 Laga, 99, avenue Jean-Baptiste Cl\'ement, F-93430 Villetaneuse, France}
\email{ddsf@math.univ-paris13.fr}
\address{Department of Mathematics, University of Chicago, 5734 University Avenue, Chicago, IL 60637-1514, USA}
\email{cek@math.uchicago.edu}
\address{Department of Mathematics and Statistics, University of Helsinki, PO Box 68, 00014 Helsinki, Finland}
\email{mikko.salo@helsinki.fi}

\date{March 30, 2011}

\begin{abstract}
   In \cite{DKSaU} anisotropic inverse problems were considered in certain admissible geometries, that is, on compact Riemannian manifolds with boundary 
    which are conformally embedded in a product of the Euclidean line and a simple manifold. In particular, it was proved that 
   a bounded smooth potential in a Schr\"odinger equation was uniquely determined by the Dirichlet-to-Neumann map in dimensions $n \geq 3$. 
   In this article we extend this result to the case of unbounded potentials, namely those in $L^{n/2}$. In the process, we derive $L^p$
   Carleman estimates with limiting Carleman weights similar to the Euclidean estimates of Jerison-Kenig \cite{JK} and Kenig-Ruiz-Sogge \cite{KRS}.
\end{abstract}
\maketitle
\setcounter{tocdepth}{1} 
\tableofcontents
%
%
\begin{section}{Introduction} \label{sec_intro}
%
In this paper we consider the problem of proving $L^p$ estimates for limiting Carleman weights on Riemannian manifolds, and the related inverse problem of recovering an $L^{n/2}$ potential from the Dirichlet-to-Neumann map (DN map) related to the Schr\"odinger equation. The main motivation comes from the inverse conductivity problem posed by Calder\'on \cite{calderon}. This problem asks to determine the conductivity function of a medium from electrical measurements made on its boundary.

In mathematical terms, if $\Omega \Subset \mR^n$ is the medium of interest having a positive conductivity coefficient $\gamma$, in the Calder\'on problem one considers the conductivity equation 
$$
\nabla \cdot \gamma \nabla u = 0 \quad \text{in } \Omega
$$
and defines the DN map by 
$$
\Lambda_{\gamma}: u|_{\partial \Omega} \mapsto \gamma \frac{\partial u}{\partial \nu}\Big|_{\partial \Omega}.
$$
This operator maps the voltage at the boundary to the current given by $\gamma$ times the normal derivative, which encodes the electrical measurements at the boundary. The inverse problem of Calder\'on asks to determine $\gamma$ from the knowledge of $\Lambda_{\gamma}$. This problem has been extensively studied and we refer to \cite{U_IP} for a recent survey.

The anisotropic Calder\'on problem considers the case where the conductivity $\gamma$ is a symmetric positive definite matrix instead of a scalar function. This corresponds to situations where the electrical properties of the medium depend on direction. The problem is open in general in dimensions $n \geq 3$, see \cite{DKSaU} for known results and more details. Following \cite{LeU} the problem may be recast as the determination of the metric $g$ on a compact Riemannian manifold $(M,g)$ with boundary from the corresponding DN map. 
In \cite{DKSaU} progress was made on the anisotropic Calder\'on problem in the following class of conformal smooth manifolds.

\begin{definition}
A compact Riemannian manifold $(M,g)$, with dimension $n \geq 3$ and with boundary $\partial M$, is called \emph{admissible} if $M \Subset \mR \times M_0$ for some $(n-1)$-dimensional simple manifold $(M_0,g_0)$, and if $g = c(e \oplus g_0)$ where $e$ is the Euclidean metric on $\mR$ and $c$ is a smooth positive function on $M$.
\end{definition}

Here, a compact manifold $(M_0,g_0)$ with boundary is \emph{simple} if for any $p \in M_0$ the exponential map $\exp_p$ with its maximal domain of definition is a diffeomorphism onto $M_0$, and if $\partial M_0$ is strictly convex (that is, the second fundamental form of $\partial M_0 \hookrightarrow M_0$ is positive definite).

In \cite{DKSaU} it was proved that a Riemannian metric in a conformal class of admissible geometries is uniquely determined by the DN map. This was obtained as a corollary of a result for the Schr\"odinger equation in a fixed admissible manifold, stating that a bounded smooth potential $q$ is determined by the corresponding DN map. In \cite{DKSaU} all coefficients were assumed infinitely differentiable. In this paper we relax this requirement and show that a complex potential $q \in L^{n/2}(M)$ is determined by the DN map.

To state the main result, assume that $(M,g)$ is a compact Riemannian manifold with smooth boundary $\partial M$, and let $\Delta_g$ be the Laplace-Beltrami operator. Given a complex function $q \in L^{n/2}(M)$, where $n \geq 3$ is the dimension of the manifold  $M$, we consider the Dirichlet problem 
\begin{equation*} 
(-\Delta_g + q)u = 0 \text{ in } M, \quad u|_{\partial M} = f.
\end{equation*}
We assume throughout that $0$ is not a Dirichlet eigenvalue for this problem, and then standard arguments (see Appendix \ref{appendix_wellposedness}) show that there is a well-defined DN map 
$$
\Lambda_{g,q}: H^{1/2}(\partial M) \to H^{-1/2}(\partial M), \ \ f \mapsto \partial_{\nu} u|_{\partial M}.
$$

The following uniqueness theorem is the main result for the inverse problem. (The assumption $q \in L^{n/2}$ may be considered optimal in the context of the standard wellposedness theory for the Dirichlet problem with $L^p$ potentials, and also for the strong unique continuation principle to hold \cite{JK}.)

\begin{thm} \label{theorem_main}
Let $(M,g)$ be admissible and let $q_1, q_2$ be complex functions in $L^{n/2}(M)$. If $\Lambda_{g,q_1} = \Lambda_{g,q_2}$, then $q_1 = q_2$.
\end{thm}

In the case where $M$ is a bounded domain in $\mR^n$ and $g$ is the Euclidean metric, this result is due to Lavine and Nachman \cite{LN} following the earlier result of Jerison and Kenig for $q_j \in L^{n/2+\eps}(M)$ for some $\eps > 0$ (see Chanillo \cite{Cha} for an account and also for a related result with $q_j$ in a Fefferman-Phong class with small norm). As mentioned above, if $q$ is a smooth function on an admissible manifold $M$ this result was proved in \cite{DKSaU} by using $L^2$ Carleman estimates. In fact, smoothness of $q$ is not essential, and by inspecting the proof of \cite{DKSaU} the uniqueness result can be extended to bounded continuous $q$ (with the complex geometrical optics construction in the proof going through for $q \in L^n(M)$). However, the proof for $q \in L^{n/2}$ requires to replace the $L^2$ Carleman estimates in \cite{DKSaU} with corresponding $L^p$ Carleman estimates.

The other main result in this paper is a $L^p$ Carleman estimate for limiting Carleman weights on Riemannian manifolds. The concept of limiting Carleman weights was introduced in \cite{ksu} as part of a general procedure for producing special complex geometrical optics solutions to elliptic equations, with applications to inverse problems. We refer to \cite{DKSaU} for a precise definition and more careful analysis of limiting Carleman weights, also on Riemannian manifolds. For present purposes it is sufficient to mention that the existence of a limiting Carleman weight on $(M,g)$ in dimensions $n \geq 3$ is locally equivalent with the manifold being admissible, and that typical limiting Carleman weights in $\mR^n$, $n \geq 3$, include the linear weight $\varphi(x) = x_1$ and logarithmic weight $\varphi(x) = \log\,\abs{x}$.

The last two weights are featured in the literature of Carleman estimates and unique continuation, in particular in the scale invariant $L^p$ Carleman estimates of Kenig-Ruiz-Sogge \cite{KRS} for the linear weight and of Jerison-Kenig \cite{JK} for the logarithmic weight. We prove an analogue of these estimates on more general Riemannian manifolds. Note that the existence of a limiting Carleman weight requires at least locally a product structure on the manifold, and therefore the following result is stated for the linear weight on a product manifold. The result, in the case when the manifold $(M_0,g_0)$ below is the standard 
$n-1$ dimensional torus, is due to Shen \cite{Shen}.

\begin{thm}
\label{Carl:LpCarlThm}
Let $(M_0,g_0)$ be an $(n-1)$-dimensional compact manifold without boundary, and equip $\R \times M_0$ with the metric $g = e \oplus g_0$ where $e$ is the Euclidean metric. The Euclidean coordinate is denoted by $x_1$. For any compact interval $I \subseteq \R$ there exists a constant $C_I > 0$ such that if $\abs{\tau} \geq 4$ and 
         $$ \tau^2 \notin {\rm Spec}(-\Delta_{g_{0}}) $$
    then we have
    \begin{align*}
         \| \e^{\tau x_{1}} u\|_{L^{\frac{2n}{n-2}}(\mR \times M_0)} \leq C_I \| \e^{\tau x_{1}} \Delta_g u\|_{L^{\frac{2n}{n+2}}(\mR \times M_0)} 
    \end{align*}
    when $u \in C^{\infty}_{0}(I \times M_{0})$.
\end{thm}

The proof of the $L^2$ Carleman estimates for limiting Carleman weights in \cite{DKSaU} is based on integration by parts and cannot be used in the $L^p$ setting. However, in \cite{KSaU} another proof of the $L^2$ Carleman estimate is given; this proof is based on Fourier analysis and gives an explicit inverse for the conjugated Laplacian. We will derive the $L^p$ bounds from this explicit inverse operator. This follows the proof of the $L^p$ Carleman estimate of Jerison-Kenig \cite{JK} using Jerison's approach \cite{Jerison}, \cite[Section 5.1]{Sogge} based on the spectral cluster estimates of Sogge \cite{Sogge}. Finally, if one allows strongly pseudoconvex Carleman weights then much stronger estimates are available (see for instance \cite{KT_nonsmooth,KT}), however for the applications to inverse problems it seems necessary to restrict to limiting Carleman weights.

\begin{remark}
The above theorems are in the setting of (conformal) product manifolds. However, the results also apply to warped products. If $f: \R \to \R$ is a smooth function and $(M_0,g_0)$ is an $(n-1)$-dimensional manifold, the warped product $\R \times_{\e^{2f}} M_{0}$ is the manifold $M=\R \times M_{0}$ endowed with the metric
$$
g(x_1,x') = \left( \begin{array}{cc} 1 & 0 \\ 0 & \e^{2f(x_1)} g_0(x') \end{array} \right).
$$
We choose coordinates $y_1 = \eta(x_1)$, $y' = x'$ for a suitable smooth strictly increasing function $\eta$. In fact, if 
$$
\eta(t) = \int_0^t \e^{-f(s)} \,\dd s
$$
then $\eta'(t)^{-2} = e^{2f(t)}$ and the metric in $y$ coordinates becomes a conformal multiple of a product metric, 
$$
g(y_1,y') = \e^{2f(\eta^{-1}(y_1))} \left( \begin{array}{cc} 1 & 0 \\ 0 & g_0(y') \end{array} \right).
$$
Warped products have a natural limiting Carleman weight $\varphi(y) = y_1$, and Theorem \ref{theorem_main} remains true in conformal multiples of warped products whenever $(M_0,g_0)$ is a simple manifold.
\end{remark}

The paper is organized as follows. Section \ref{sec_intro} is the introduction. In Section \ref{sec_Lp_carleman} we prove the $L^p$ Carleman estimate complemented with the usual $L^2$ Carleman estimates. Section \ref{sec_cgo} presents the construction of complex geometrical optics solutions for Schr\"odinger equations with $L^{n/2}$ potentials in admissible geometries. The proof of Theorem \ref{theorem_main} is contained in Section \ref{sec_proof_main}, modulo a uniqueness result for an analogue of the attenuated geodesic ray transform acting on unbounded functions. This last result has a different character than the rest of the proof, and it is therefore established separately in Section \ref{sec_raytransform}. There are two appendices concerning the wellposedness of the Dirichlet problem and the normal operator for the attenuated ray transform.

\subsection*{Acknowledgements}

The last named author would like to thank Adrian Nachman for explaining his unpublished argument with Richard Lavine \cite{LN} which proves a uniqueness result for $L^{n/2}$ potentials in Euclidean space. C.K. is supported partly by NSF grant DMS-0968472, and M.S. is supported in part by the Academy of Finland.
D. DSF. would like to thank the University of Chicago for its hospitality.

\end{section}
%
%
\begin{section}{$L^p$ Carleman estimates} \label{sec_Lp_carleman}
The aim of this section is to prove Theorem \ref{Carl:LpCarlThm}, which is an analogue of the $L^p$ Carleman estimates obtained in the Euclidean case by Jerison and Kenig \cite{JK}
(for logarithmic weights) or by Kenig, Ruiz and Sogge \cite{KRS} (for linear weights). In fact, we prove a more general result which implies Theorem \ref{Carl:LpCarlThm} by taking $f = \e^{\tau x_1} \Delta_g \e^{-\tau x_1} u$ for $u \in C^{\infty}_0(I \times M_0)$. The case when $(M_0,g_0)$ is the standard 
$n-1$ dimensional torus is due to Shen \cite{Shen}.

\begin{prop} \label{claim_gtau}
Let $I \subseteq \R$ be a compact interval and $(M_0,g_0)$ a compact $(n-1)$-dimensional manifold without boundary. Equip $N = I \times M_0$ with the product metric $g = e \oplus g_0$. For $\abs{\tau} \geq 4$ with $\tau^2 \notin \text{Spec}(-\Delta_{g_0})$, there is a linear operator $G_{\tau}: L^2(N) \to H^2(N)$ such that 
\begin{eqnarray*}
 & \e^{\tau x_1} (-\Delta_g) \e^{-\tau x_1} G_{\tau} v = v & \text{for } v \in L^2(N), \\
 & G_{\tau} \e^{\tau x_1} (-\Delta_g) \e^{-\tau x_1} v = v & \text{for } v \in C^{\infty}_0(N^{\text{int}}).
\end{eqnarray*}
This operator satisfies 
\begin{align*}
\norm{G_{\tau} f}_{L^2(N)} &\leq \frac{C_0}{\abs{\tau}} \norm{f}_{L^2(N)}, \\
\norm{G_{\tau} f}_{H^1(N)} &\leq C_0 \norm{f}_{L^2(N)}, \\
\norm{G_{\tau} f}_{L^{\frac{2n}{n-2}}(N)} &\leq C_0 \norm{f}_{L^{\frac{2n}{n+2}}(N)},
\end{align*}
where $C_0$ is independent of $\tau$ (but may depend on $I$).
\end{prop}

\begin{remark}
     In the Euclidean case, $L^p$ Carleman estimates with linear weights can be obtained from $L^p$ Carleman estimates with 
     pseudoconvex Carleman weights by scaling. Indeed, suppose that the following Carleman estimate
          $$  \big\| \e^{\tau (x_{1}+x_{1}^2/2\eps)} u \big\|_{L^{\frac{2n}{n-2}}(\R^n)} \leq 
                C_{K}  \big\| \e^{\tau (x_{1}+x_{1}^2/2\eps)} \Delta u\|_{L^{\frac{2n}{n+2}}(\R^n)}, $$
     holds for all $\eps \leq \eps_{0}$ and all $u \in \D(K)$, then applying this estimate to $u_{\mu}=u(\mu \, \cdot)$ with $\mu \geq 1$ and $u \in \D(K)$, one gets
          $$  \big\|\e^{\frac{\tau}{\mu} x_{1}+\frac{\tau}{\mu^2} x_{1}^2/2\eps} u\big\|_{L^{\frac{2n}{n-2}}(\R^n)} 
                \leq C_{K}  \big\|\e^{\frac{\tau}{\mu} x_{1}+\frac{\tau}{\mu^2} x_{1}^2/2\eps} \Delta u\big\|_{L^{\frac{2n}{n+2}}(\R^n)}. $$
     Choosing $\mu=\sqrt{\tau}$, and using the fact that $\e^{x_{1}^2/2\eps} \simeq C_{\eps}$ on $K$, one gets the
     Carleman estimate
           $$  \| \e^{\mu x_{1}} u\|_{L^{\frac{2n}{n-2}}(\R^n)} \leq C_{K,\eps}  \|  \e^{\mu x_{1}} \Delta u\|_{L^{\frac{2n}{n+2}}(\R^n)}, $$
     for all $u \in \D(K)$. However, in the anisotropic case, one has to find another way.
\end{remark}

To prepare for the proof of Proposition \ref{claim_gtau} consider the Laplace-Beltrami operator on $N$,
\begin{align*}
     P = \Delta_{g}=\d_{x_{1}}^2+\Delta_{g_{0}}
\end{align*}
and the corresponding conjugated operator (by the limiting Carleman weight $x_{1}$)
\begin{align}
\label{Carl:ConjOp}
     \e^{\tau x_{1}}P\e^{-\tau x_{1}} = \d_{x_{1}}^2 - 2 \tau \d_{x_{1}} + \tau^2 + \Delta_{g_{0}}.
\end{align}

We denote by $\lambda_{0}=0<\lambda_{1} \leq \lambda_{2} \leq \dots$ the sequence of
eigenvalues of $-\Delta_{g_{0}}$ on $M_{0}$ and $(\psi_{j})_{j \geq 0}$ the corresponding sequence of eigenfunctions forming an orthonormal basis of $L^2(M_0)$, 
\begin{align*}
     -\Delta_{g_{0}}\psi_{j} = \lambda_{j} \psi_{j}.
\end{align*}
We denote by $\pi_{j} : L^2(M_{0}) \to L^2(M_{0}), u \mapsto (u,\psi_j) \psi_j$ the projection on the linear space spanned by the eigenfunction $\psi_{j}$ so that 
\begin{align*}
     \sum_{j=0}^{\infty} \pi_{j} = {\rm Id}, \quad \sum_{j=0}^{\infty} \lambda_{j} \pi_{j} = -\Delta_{g_{0}}
\end{align*}
and by 
     $$ \widehat{u}(j)= \int_{M_{0}} u \, \overline{\psi_{j}} \, \dd V_{g_{0}} $$ 
the corresponding Fourier coefficients of a function $u$ on $M_{0}$. We define the spectral clusters as
\begin{align}
     \chi_{k} = \sum_{k \leq \sqrt{\lambda_{j}} < k+1} \pi_{j}, \quad k \in \N.
\end{align}
Note that these are projection operators, $\chi_{k}^2=\chi_{k}$, and they constitute a decomposition of the identity
\begin{align}
\label{Carl:Identity}
     {\rm Id} = \sum_{k=0}^{\infty} \chi_{k}.
\end{align}
We end this paragraph by recalling the spectral cluster estimates of Sogge \cite{Sogge_cluster,Sogge} that we will need:
\begin{align}
\label{Carl:spclusterEst}
     \|\chi_{k}u\|_{L^{\frac{2n}{n-2}}(M_{0})}&\leq C(1+k)^{\frac{1}{2}-\frac{1}{n}} \|u\|_{L^{2}(M_{0})}, \\ \nonumber
     \|\chi_{k}u\|_{L^{2}(M_{0})} &\leq C(1+k)^{\frac{1}{2}-\frac{1}{n}} \|u\|_{L^{\frac{2n}{n+2}}(M_{0})}.
\end{align}
The first estimate is  given in \cite[Corollary 5.1.2]{Sogge} and the second one is a consequence of the first one by duality.

\begin{proof}[Proof of Proposition \ref{claim_gtau}]
Recall that our main goal is to prove 
\begin{align}
\label{Carl:ConjEst}
     \|u\|_{L^{\frac{2n}{n-2}}(\R \times M_{0})} \leq C_{I} \|f\|_{L^{\frac{2n}{n+2}}(\R \times M_{0})}
\end{align}
when $u \in \D(I \times M_{0})$ and
\begin{align}
\label{Carl:ConjEq}
    D_{x_{1}}^2u+2i\tau D_{x_{1}}u-\Delta_{g_{0}}u-\tau^2 u = f
\end{align}
with $D_{x_{1}}=-i\partial_{x_{1}}$. The inverse operator in \eqref{Carl:ConjEq} is actually easy to write down, as was done in \cite{KSaU}. 
The same procedure appears in \cite{Jerison} and \cite[Section 5.1]{Sogge}. 
Writing $f = \sum_{j=0}^{\infty} \pi_{j}f$ and similarly for $u$, the equation formally becomes 
\begin{align*}
     (D_{x_{1}}^2 + 2i\tau D_{x_{1}} - \tau^2 + \lambda_j) \pi_{j}u = \pi_{j}f
\end{align*}
for $x_1$ on the real line and for $j \geq 0$. The symbol of the operator on the left is $\xi_1^2 + 2 i \tau \xi_1 - \tau^2 + \lambda_j$, and
this is always nonzero provided that $\tau^2 \neq \lambda_j$ for all $j$. Thus, if 
\begin{equation*}
     \tau^2 \notin \text{Spec}(-\Delta_{g_{0}}),
\end{equation*}
an inverse operator may be obtained as 
\begin{equation*}
     \tilde{G}_{\tau} f(x_1,x') =  \sum_{j=0}^{\infty} \int_{-\infty}^{\infty} m_{\tau}\big(x_{1}-y_{1},\sqrt{\lambda_{j}}\big) \,  \pi_{j}f(y_{1},x') \, \dd y_{1} 
\end{equation*}
where 
\begin{equation*}
     m_{\tau}(t,\mu) = \frac{1}{2\pi} \int_{-\infty}^{\infty} \frac{\e^{it\eta}}{\eta^2+2i\tau \eta -\tau^2 + \mu^2} \, \dd\eta, \quad \mu>0.
\end{equation*}

The operator $\tilde{G}_{\tau}$ is the same as $G_{\tau}$ in \cite[Section 4]{KSaU}, except that in the present setting $\{ \psi_j \}$ is a basis of $L^2(M_0)$ on a compact manifold $(M_0,g_0)$ without boundary instead of being a basis of Dirichlet eigenfunctions on a compact manifold with boundary. Let 
$$
L^2_{\delta}(\R \times M_0) = \big\{ f \in L^2_{\text{loc}}(\R \times M_0) \,;\, (1+x_1^2)^{\delta/2} f \in L^2(\R \times M_0) \big\}
$$
and let $H^s_{\delta}(\R \times M_0)$ by the corresponding Sobolev spaces. The proof of \cite[Proposition 4.1]{KSaU} goes through for $\tilde{G}_{\tau}$ without changes and shows that for any fixed $\delta > 1/2$, if $\abs{\tau} \geq 1$ and $\tau^2 \notin \text{Spec}(-\Delta_{g_0})$ then the equation 
$$
\e^{\tau x_1} (-\Delta_g) \e^{-\tau x_1} v = f
$$
has a unique solution $v = \tilde{G}_{\tau} f \in H^1_{-\delta}(\R \times M_0)$ for any $f \in L^2_{\delta}(\R \times M_0)$. Further, $v \in H^2_{-\delta}(\R \times M_0)$ and 
$$
\norm{v}_{H^s_{-\delta}(\R \times M_0)} \leq C_0 \abs{\tau}^{s-1} \norm{f}_{L^2_{\delta}(\R \times M_0)}, \quad 0 \leq s \leq 2.
$$

We define 
$$
G_{\tau} f(x_1,x') = \chi(x_1) \tilde{G}_{\tau} f(x_1,x')
$$
with $\chi \in C^{\infty}_{0}(\R)$ which equals $1$ on $I$. It is then clear that all the statements in the proposition except for the $L^p$ estimate follow from the results for $\tilde{G}_{\tau}$ explained above.

It is sufficient to prove the $L^p$ estimate in the case where $\tau \geq 4$ and $\tau^2 \notin \text{Spec}(-\Delta_{g_0})$. We first record a lemma.

\begin{lemma} \label{lemma:mtau}
     If $\tau > 0$, $\mu > 0$, $\tau \neq \mu$ and $t \in \mR$ then 
     \begin{equation*}
         \abs{m_{\tau}(t,\mu)} \leq \frac{1}{\mu} \e^{-\abs{\tau-\mu} \abs{t}}.
     \end{equation*}
     Besides, if $\tau>0$ then
     \begin{equation*}
         \abs{m_{\tau}(t,0)} \leq |t|\e^{-\tau\abs{t}}.
     \end{equation*}
\end{lemma}
\begin{proof}
     This follows by writing 
     \begin{align*} 
         \frac{1}{(i\eta - (\tau+\mu))(i\eta - (\tau-\mu))} = \frac{1}{2\mu} \left[ \frac{1}{i\eta - (\tau+\mu)} - \frac{1}{i\eta - (\tau-\mu)} \right]
     \end{align*}
     and by noting that for $\alpha > 0$ 
     \begin{equation*}
          \mF_{\eta}^{-1} \left\{ \frac{1}{i\eta+\alpha} \right\}(t) = \left\{ \begin{array}{ll} 0, & t < 0 \\ \e^{-\alpha t}, & t > 0, \end{array} \right.
     \end{equation*}
     and similarly for $\alpha < 0$.
     
     Furthermore we have
     \begin{equation*}
          \mF_{\eta}^{-1} \left\{ \frac{1}{(\eta+i\tau)^2} \right\}(t) = \left\{ \begin{array}{ll} t \e^{-\tau |t|}, & t < 0 \\  0, & t > 0, \end{array} \right.
     \end{equation*}
     and this concludes the proof of the lemma.
\end{proof}
From the decomposition \eqref{Carl:Identity}, the spectral cluster estimate \eqref{Carl:spclusterEst}, and the fact that spectral clusters are projections ($\chi_{k}^2=\chi_{k}$),
we get the following string of estimates
\begin{align}
\label{Carl:ClusterLpL2}
     \|u\|_{L^{\frac{2n}{n-2}}(M_{0})} &= \bigg\|\sum_{k=0}^{\infty} \chi^2_{k}u\bigg\|_{L^{\frac{2n}{n-2}}(M_{0})} \\ \nonumber
     &\lesssim \sum_{k=0}^{\infty} (1+k)^{\frac{1}{2}-\frac{1}{n}}  \|\chi_{k}u\|_{L^{2}(M_{0})}.
\end{align}
Since
\begin{align*}
     \|\chi_{k}u\|_{L^{2}(M_{0})} = \bigg(\sum_{k \leq \sqrt{\lambda_{j}}<k+1} |\widehat{u}(j)|^2\bigg)^{\frac{1}{2}}
\end{align*}
if we apply the inequality \eqref{Carl:ClusterLpL2} to $u=G_{\tau}f(x_{1},\cdot)$, we get for almost every $x_1 \in I$ 
\begin{multline*}
     \|G_{\tau}f(x_{1},\cdot)\|_{L^{\frac{2n}{n-2}}(M_{0})}  \lesssim \sum_{k=0}^{\infty} (1+k)^{\frac{1}{2}-\frac{1}{n}}  \\
     \times \bigg(\sum_{k \leq \sqrt{\lambda_{j}}<k+1} \bigg|\int_{-\infty}^{\infty}m_{\tau}\big(x_{1}-y_{1},\sqrt{\lambda_{j}}\big)
     \widehat{f}(y_{1},j)  \, \dd y_{1}\bigg|^2 \bigg)^{\frac{1}{2}}.
\end{multline*}
By Minkowski's inequality, we have
\begin{multline*}
     \|G_{\tau}f(x_{1},\cdot)\|_{L^{\frac{2n}{n-2}}(M_{0})}  \lesssim \sum_{k=0}^{\infty} (1+k)^{\frac{1}{2}-\frac{1}{n}}  \\
     \times \int_{-\infty}^{\infty} \bigg(\sum_{k \leq \sqrt{\lambda_{j}}<k+1} \Big|m_{\tau}\big(x_{1}-y_{1},\sqrt{\lambda_{j}}\big)
     \widehat{f}(y_{1},j)\Big|^2  \bigg)^{\frac{1}{2}} \, \dd y_{1}
\end{multline*}
and since
\begin{align*}
    &\sum_{k \leq \sqrt{\lambda_{j}}<k+1} \Big|m_{\tau}\big(x_{1}-y_{1},\sqrt{\lambda_{j}}\big)
     \widehat{f}(y_{1},j)\Big|^2  \\ &\qquad \quad \leq \sup_{k \leq \sqrt{\lambda_{j}}<k+1} \big|m_{\tau}(x_{1}-y_{1},\sqrt{\lambda_{j}}\big)\big|^2
     \sum_{k \leq \sqrt{\lambda_{j}}<k+1} \big|\widehat{f}(y_{1},j)\big|^2  \\
     &\qquad \quad \leq \sup_{k \leq \sqrt{\lambda_{j}}<k+1} \big|m_{\tau}(x_{1}-y_{1},\sqrt{\lambda_{j}}\big)\big|^2 \times
     \big\|\chi_{k}f(y_{1}, \cdot)\|_{L^2(M_{0})}^2
\end{align*}
using once again the spectral cluster estimate  \eqref{Carl:spclusterEst}, we finally get 
\begin{multline}
\label{Carl:GtauEst}
     \|G_{\tau}f(x_{1},\cdot)\|_{L^{\frac{2n}{n-2}}(M_{0})}  \lesssim \sum_{k=0}^{\infty} (1+k)^{1-\frac{2}{n}}  \\
     \times \int_{-\infty}^{\infty}  \sup_{k \leq \sqrt{\lambda_{j}}<k+1} \big|m_{\tau}(x_{1}-y_{1},\sqrt{\lambda_{j}}\big)\big| \times
     \big\|f(y_{1}, \cdot)\|_{L^{\frac{2n}{n+2}}(M_{0})} \, \dd y_{1}.
\end{multline}
Using Lemma \ref{lemma:mtau}, we estimate
\begin{align*}
     \sup_{k \leq \sqrt{\lambda_{j}}<k+1} \big|m_{\tau}(t,\sqrt{\lambda_{j}}\big)\big|
     \leq \frac{1}{k} \begin{cases} \e^{-(k-\tau) |t|} & \text{ when } \tau < k \\ 1 & \text{ when } k\leq \tau <k+1   
     \\ \e^{-(\tau-k-1)|t|} & \text{ when } \tau \geq k+1 \end{cases}
\end{align*}           
with $k>0$. (Note that when $k=0$, a majorant is $\e^{-(\tau/2)|t|}$ for $\tau \geq 4$). This allows us to estimate the series
\begin{align*}
    \sum_{k=0}^{\infty} (1&+k)^{1-\frac{2}{n}} \sup_{k \leq \sqrt{\lambda_{j}}<k+1} \big|m_{\tau}(t,\sqrt{\lambda_{j}}\big)\big| \\ \nonumber
    &\lesssim  \sum_{1 \leq k \leq \tau-2} k^{-\frac{2}{n}}\e^{-(\tau-k-1) |t|}+ \tau^{-\frac{2}{n}} +  \sum_{ k > \tau+1} k^{-\frac{2}{n}}\e^{-(k-\tau)|t|} + e^{-(\tau/2)\abs{t}} \\ \nonumber
    &\lesssim \int_{0}^{\tau-2} r^{-\frac{2}{n}} \e^{-(\tau-r-2) |t|} \, \dd r + 1 + \int_{\tau}^{\infty} r^{-\frac{2}{n}} \e^{-(r-\tau)|t|} \, \dd r.
\end{align*}
By an obvious change of variables we have
\begin{align*}
    & \int_{0}^{\tau-2} r^{-\frac{2}{n}} \e^{-(\tau-r-2) |t|} \, \dd r + \int_{\tau}^{\infty} r^{-\frac{2}{n}} \e^{-(r-\tau)|t|} \, \dd r \\
    & \qquad \leq 2 |t|^{-1+\frac{2}{n}} \bigg( \int_{0}^1 r^{-\frac{2}{n}} \, \dd r + \int_{1}^{\infty} \e^{-s} \, \dd s\bigg)
     \lesssim |t|^{-1+\frac{2}{n}} 
\end{align*}
whence
\begin{align}
\label{Carl:SeriesEst}
    \sum_{k=0}^{\infty} (1&+k)^{1-\frac{2}{n}} \sup_{k \leq \sqrt{\lambda_{j}}<k+1} \big|m_{\tau}(t,\sqrt{\lambda_{j}}\big)\big|   \lesssim 1+ |t|^{-1+\frac{2}{n}} .
\end{align}
From the estimates \eqref{Carl:GtauEst} and \eqref{Carl:SeriesEst}, we obtain
\begin{align*}
     \|G_{\tau}&f(x_{1},\cdot)\|_{L^{\frac{2n}{n-2}}(M_{0})} \\ & \lesssim \int_{-\infty}^{\infty} \Big(1+|x_{1}-y_{1}|^{-1+\frac{2}{n}}\Big)
     \big\|f(y_{1}, \cdot)\|_{L^{\frac{2n}{n+2}}(M_{0})} \, \dd y_{1}  \\
      &\lesssim \int_{-\infty}^{\infty} |x_{1}-y_{1}|^{-1+\frac{2}{n}} \big\|f(y_{1}, \cdot)\|_{L^{\frac{2n}{n+2}}(M_{0})} \, \dd y_{1}
      + |I|^{\frac{1}{2}-\frac{1}{n}} \|f\|_{L^{\frac{2n}{n+2}}(I \times M_0)}
\end{align*}
and we conclude using the Hardy-Littlewood-Sobolev inequality
\begin{align*}
     \|G_{\tau}f\|_{L^{\frac{2n}{n-2}}(I \times M_{0})}  \lesssim \big\|f\|_{L^{\frac{2n}{n+2}}(I \times M_{0})}.
\end{align*}
This completes the proof of Proposition\ref{claim_gtau}.
\end{proof}
\end{section}
%
%
\begin{section}{Complex geometrical optics} \label{sec_cgo}
In this section we will construct the complex geometrical optics solutions that will be used to recover an $L^{n/2}$ potential. Throughout the section, let $(M,g)$ be a compact manifold with smooth boundary, and let $(M,g) \Subset (T,g)  \Subset (\tilde{T},g)$ where $T = \mR \times M_0$, $\tilde{T} = \mR \times \tilde{M_0}$, and $g = e \oplus g_0$, and $(M_0,g_0) \Subset (\tilde{M}_0,g_0)$ are two $(n-1)$-dimensional simple manifolds. We also assume that $n \geq 3$.

First we state a consequence of Proposition \ref{claim_gtau} for the manifold $M$ (this follows easily by embedding $(\tilde{M}_0,g_0)$ in some compact manifold without boundary and using suitable restrictions and extensions by zero).

\begin{prop} \label{claim_gtau2}
For $\abs{\tau} \geq 4$ outside a countable set, there is a linear operator $G_{\tau}: L^2(M) \to H^2(M)$ such that 
\begin{eqnarray*}
 & \e^{\tau x_1} (-\Delta_g) \e^{-\tau x_1} G_{\tau} v = v & \text{for } v \in L^2(M), \\
 & G_{\tau} \e^{\tau x_1} (-\Delta_g) \e^{-\tau x_1} v = v & \text{for } v \in C^{\infty}_0(M^{\text{int}}).
\end{eqnarray*}
This operator satisfies 
\begin{align*}
\norm{G_{\tau} f}_{L^2(M)} &\leq \frac{C_0}{\abs{\tau}} \norm{f}_{L^2(M)}, \\
\norm{G_{\tau} f}_{H^1(M)} &\leq C_0 \norm{f}_{L^2(M)}, \\
\norm{G_{\tau} f}_{L^{\frac{2n}{n-2}}(M)} &\leq C_0 \norm{f}_{L^{\frac{2n}{n+2}}(M)},
\end{align*}
where $C_0$ is independent of $\tau$.
\end{prop}

Let us first construct the required complex geometrical optics solutions for the case where no potential is present. This is analogous to \cite[Proposition 5.1]{DKSaU} for $q = 0$.

\begin{prop} \label{prop_cgo_free}
Let $\omega \in \tilde{M}_0 \setminus M_0$ be a fixed point, let $\lambda \in \mR$ be fixed, and let $b \in C^{\infty}(S^{n-2})$ be a function. Write $x = (x_1,r,\theta)$ where $(r,\theta)$ are polar normal coordinates with center $\omega$ in $(\tilde{M}_0,g_0)$. For $\abs{\tau}$ sufficiently large outside a countable set, there exists $u_0 \in H^1(M)$ satisfying 
\begin{align*}
-\Delta_g u_0 &= 0 \ \ \text{in } M, \\ u_0 &= \e^{-\tau x_1}(\e^{-i\tau r} \abs{g}^{-1/4} \e^{i\lambda(x_1+ir)} b(\theta) + r_0)
\end{align*}
where $r_0$ satisfies 
\begin{equation*}
\abs{\tau} \norm{r_0}_{L^2(M)} + \norm{r_0}_{H^1(M)} + \norm{r_0}_{L^{\frac{2n}{n-2}}(M)} \lesssim 1.
\end{equation*}
\end{prop}
\begin{proof}
The claim follows if one can find $r_0$ satisfying 
\begin{equation*}
\e^{\tau x_1} (-\Delta_g) \e^{-\tau x_1} r_0 = f
\end{equation*}
with the required norm estimates, where 
   $$ f = \e^{\tau x_1} \Delta_g \e^{-\tau x_1}(\e^{-i\tau r} \e^{i\lambda(x_1+ir)} b(\theta)). $$
It is enough to take $r_0 = G_{\tau} f$ and to note that \begin{equation*}
\norm{f}_{L^2(M)} = \norm{\Delta_g(\e^{i\lambda(x_1+ir)} b(\theta))}_{L^2(M)} \lesssim 1.
\end{equation*}
The $L^2$ and $H^1$ estimates follow from Proposition \ref{claim_gtau2}. The $L^{\frac{2n}{n-2}}$ estimate follows from the $H^1$ estimate and Sobolev embedding, or alternatively from the $L^{\frac{2n}{n+2}} \to L^{\frac{2n}{n-2}}$ estimate for $G_{\tau}$.
\end{proof}

We next consider the case with a potential $q \in L^{n/2}(M)$, and try to find a solution to $(-\Delta_g + q)u = 0$ in $M$ of the form 
\begin{equation*}
u = u_0 + \e^{-\tau x_1} r_1.
\end{equation*}
Since $-\Delta_g u_0 = 0$, the function $r_1$ should satisfy 
\begin{equation} \label{r1_equation}
\e^{\tau x_1} (-\Delta_g + q) \e^{-\tau x_1} r_1 = - q \e^{\tau x_1} u_0.
\end{equation}
Since $q$ is only in $L^{n/2}(M)$, here we need to use the $L^{\frac{2n}{n+2}} \to L^{\frac{2n}{n-2}}$ estimates for $G_{\tau}$. We follow the argument of Lavine and Nachman \cite{LN}. It will be convenient to symmetrize the situation slightly. Later on, the $L^n$ functions $m_j$ in the next lemma are chosen to be essentially $\abs{q}^{1/2}$.

\begin{lemma} \label{lemma_m1_gtau_m2}
Let $m_1, m_2 \in L^n(M)$ be two fixed functions. Then for $\abs{\tau} \geq \tau_0$ outside a countable set, 
\begin{equation} \label{m1_gtau_m2_estimate}
\norm{m_1 G_{\tau} m_2 f}_{L^2} \leq C_0 \norm{m_1}_{L^n} \norm{m_2}_{L^n} \norm{f}_{L^2}.
\end{equation}
Further, 
\begin{equation} \label{m1_gtau_m2_convergence_estimate}
\norm{m_1 G_{\tau} m_2 f}_{L^2(M) \to L^2(M)} \to 0
\end{equation}
as $\abs{\tau} \to \infty$.
\end{lemma}
\begin{proof}
The H\"older inequality and Proposition \ref{claim_gtau2} imply that 
\begin{align*}
\norm{m_1 G_{\tau} m_2 f}_{L^2} &\leq \norm{m_1}_{L^n} \norm{G_{\tau} m_2 f}_{L^{\frac{2n}{n-2}}} \leq C_0 \norm{m_1}_{L^n} \norm{m_2 f}_{L^{\frac{2n}{n+2}}} \\
 &\leq C_0 \norm{m_1}_{L^n} \norm{m_2}_{L^n} \norm{f}_{L^2}.
\end{align*}
Let $\eps > 0$ and decompose $m_j = m_j^{\sharp} + m_j^{\flat}$ where $m_j^{\sharp} \in L^{\infty}(M)$,
\begin{align*}
\norm{m_j^{\sharp}}_{L^n} &\leq \norm{m_j}_{L^n} \\ \text{ and } \quad
\norm{m_j^{\flat}}_{L^n} &\leq \frac{\eps}{3 C_0 \max(\norm{m_1}_{L^n}, \norm{m_2}_{L^n})}.
\end{align*}
(One can take for instance $m_j^{\sharp} = m_j \chi_{\{\abs{m_j} \leq \mu\}}$ for large enough $\mu$.) It follows from the $L^2$ estimates for $G_{\tau}$ and \eqref{m1_gtau_m2_estimate} that 
\begin{align*}
\norm{m_1 G_{\tau} m_2 f}_{L^2} &\leq \norm{m_1^{\sharp} G_{\tau} m_2^{\sharp} f}_{L^2} + \norm{m_1^{\sharp} G_{\tau} m_2^{\flat} f}_{L^2} + \norm{m_1^{\flat} G_{\tau} m_2 f}_{L^2} \\
 &\leq \left( \frac{C_0 \norm{m_1^{\sharp}}_{L^{\infty}} \norm{m_2^{\sharp}}_{L^{\infty}}}{\abs{\tau}} + \frac{\eps}{3} + \frac{\eps}{3} \right) \norm{f}_{L^2}.
\end{align*}
The last expression is bounded by $\eps \norm{f}_{L^2}$ if $\abs{\tau}$ is sufficiently large. This proves \eqref{m1_gtau_m2_convergence_estimate}.
\end{proof}

We now finish the construction of complex geometrical optics solutions.

\begin{prop} \label{prop_cgo_main}
Assume that $q \in L^{n/2}(M)$. Let $\omega \in \tilde{M}_0 \setminus M_0$ be a fixed point, let $\lambda \in \mR$ be fixed, and let $b \in C^{\infty}(S^{n-2})$ be a function. Write $x = (x_1,r,\theta)$ where $(r,\theta)$ are polar normal coordinates with center $\omega$ in $(\tilde{M}_0,g_0)$. For $\abs{\tau}$ sufficiently large outside a countable set, there exists a solution $u \in H^1(M)$ of $(-\Delta_g + q) u = 0$ in $M$ of the form 
\begin{equation*}
u = \e^{-\tau x_1}(\e^{-i\tau r} \abs{g}^{-1/4} \e^{i\lambda(x_1+ir)} b(\theta) + \tilde{r})
\end{equation*}
where $\tilde{r}$ satisfies 
\begin{equation*}
\norm{\tilde{r}}_{L^{\frac{2n}{n-2}}(M)} \lesssim 1, \qquad \norm{\tilde{r}}_{L^2(M)} \to 0 \text{ as } \abs{\tau} \to \infty.
\end{equation*}
\end{prop}
\begin{proof}
As explained above, we let $u_0$ be the harmonic function given in Proposition \ref{prop_cgo_free}, and look for a solution of the form $u = u_0 + \e^{-\tau x_1} r_1$. We write $q(x) = \abs{q(x)} \e^{i\alpha(x)} = \abs{q(x)}^{1/2} m(x)$ where $m(x) = \abs{q(x)}^{1/2} \e^{i\alpha(x)}$. Then $\abs{q}^{1/2}, m \in L^n(M)$ with $L^n$ norms equal to $\norm{q}_{L^{n/2}}^{1/2}$.

We obtain a solution $u$ provided that \eqref{r1_equation} holds. Trying $r_1$ in the form $r_1 = G_{\tau} \abs{q}^{1/2} v$, we see that $v$ should satisfy 
\begin{equation*}
(\id + m G_{\tau} \abs{q}^{1/2}) v = - m \e^{\tau x_1} u_0.
\end{equation*}
By Lemma \ref{lemma_m1_gtau_m2}, for $\abs{\tau}$ sufficiently large one has $\norm{m G_{\tau} \abs{q}^{1/2}}_{L^2 \to L^2} \leq 1/2$. One then obtains a solution 
\begin{equation*}
v = - (\id + m G_{\tau} \abs{q}^{1/2})^{-1} (m \e^{\tau x_1} u_0).
\end{equation*}
Since $\norm{m \e^{\tau x_1} u_0}_{L^2} \leq \norm{m}_{L^n} \norm{\e^{\tau x_1} u_0}_{L^{\frac{2n}{n-2}}} \lesssim 1$, it follows that $\norm{v}_{L^2} \lesssim 1$. Consequently 
\begin{equation*}
\norm{r_1}_{L^{\frac{2n}{n-2}}} \leq C_0 \norm{\abs{q}^{1/2} v}_{L^{\frac{2n}{n+2}}} \lesssim 1.
\end{equation*}
Now $u$ is of the form given in the statement of the proposition, provided that 
\begin{equation*}
\tilde{r} = r_0 + r_1.
\end{equation*}
This remainder term satisfies $\norm{\tilde{r}}_{L^{\frac{2n}{n-2}}} \lesssim 1$.

To study $\norm{\tilde{r}}_{L^2}$ we fix $\eps > 0$ and make a decomposition $\abs{q}^{1/2} = s^{\sharp} + s^{\flat}$ where $s^{\sharp} \in L^{\infty}(M)$, $\norm{s^{\sharp}}_{L^n} \leq \norm{q}_{L^{n/2}}^{1/2}$, and $\norm{s^{\flat}}_{L^n} \leq \eps$. Then 
\begin{align*}
\norm{r_1}_{L^2} &\leq \norm{G_{\tau} s^{\sharp} v}_{L^2} + C_1 \norm{G_{\tau} s^{\flat} v}_{L^{\frac{2n}{n-2}}} \\
 &\leq \left( \frac{C_0 \norm{s^{\sharp}}_{L^{\infty}}}{\abs{\tau}} + C_0 C_1 \norm{s^{\flat}}_{L^n} \right) \norm{v}_{L^2}.
\end{align*}
Choosing $\abs{\tau}$ sufficiently large, we see that $\norm{r_1}_{L^2} \lesssim \eps$ for $\abs{\tau}$ large. Since also $\norm{r_0}_{L^2(M)} \lesssim \abs{\tau}^{-1}$, it follows that $\norm{\tilde{r}}_{L^2} \to 0$ as $\abs{\tau} \to \infty$.

Finally, to prove that $u \in H^1(M)$, it is enough to consider a compact manifold $(\hat{M},g)$ which is slightly larger than $(M,g)$ and extend $q$ by zero outside $M$, and to perform the above construction of solutions in $\hat{M}$. One obtains a solution $u \in L^{\frac{2n}{n-2}}(\hat{M}) \subseteq L^2(\hat{M})$, and $\Delta_g u = q u \in L^{\frac{2n}{n+2}}(\hat{M}) \subseteq H^{-1}(\hat{M})$ by Sobolev embedding. Elliptic regularity implies that $u \in H^1_{\text{loc}}(\hat{M}^{\text{int}})$, thus also $u \in H^1(M)$.
\end{proof}

\end{section}
%
%
\begin{section}{Recovering the potential} \label{sec_proof_main}

We are now ready to give the proof of the main uniqueness result.

\begin{proof}[Proof of Theorem \ref{theorem_main}]
Assume, as before, that $(M,g) \Subset (T,g) \Subset (\tilde{T},g)$ where $T = \mR \times M_0$, $\tilde{T} = \mR \times \tilde{M}_0$, and $(M_0,g_0) \Subset (\tilde{M}_0,g_0)$ are two $(n-1)$-dimensional simple manifolds where $n \geq 3$. Also assume that $g = e \oplus g_0$.

From the assumption $\Lambda_{g,q_1} = \Lambda_{g,q_2}$, writing $q = q_1-q_2$, we know from Lemma \ref{lemma_integral_identity} that 
\begin{equation} \label{q_integral_identity}
\int_M q u_1 u_2 \, \dd V_g = 0
\end{equation}
where $u_1, u_2 \in H^1(M)$ are solutions of $(-\Delta_g + q_1) u_1 = 0$ in $M$ and $(-\Delta_g + q_2) u_2 = 0$ in $M$. By Proposition \ref{prop_cgo_main}, for $\tau$ sufficiently large outside a countable set there exist solutions $u_j$ of the form 
\begin{align*}
u_1 &= \e^{-\tau (x_1+ir)}(\abs{g}^{-1/4} \e^{i\lambda(x_1+ir)} b(\theta) + r_1), \\
u_2 &= \e^{\tau (x_1+ir)}(\abs{g}^{-1/4} + r_2).
\end{align*}
Here $\lambda$ is a fixed real number, $b \in C^{\infty}(S^{n-2})$ is a fixed function, and $x = (x_1,r,\theta)$ are coordinates in $\tilde{T}$ where $(r,\theta)$ are polar normal coordinates in $(\tilde{M}_0,g_0)$ with center at a fixed point $\omega \in \tilde{M}_0 \setminus M_0$. Also, the remainder terms satisfy 
\begin{equation*}
\norm{r_j}_{L^{\frac{2n}{n-2}}(M)} = O(1), \quad \norm{r_j}_{L^2(M)} = o(1)
\end{equation*}
as $\tau \to \infty$.

Inserting the solutions in \eqref{q_integral_identity} and noting that $\dd V_g = \abs{g}^{1/2} \dd x_1 \,\dd r \,\dd \theta$, we obtain that 
\begin{equation} \label{q_integral_identity_2}
\int_M q \e^{i\lambda(x_1+ir)} b(\theta) \, \dd x_1 \,\dd r \, \dd \theta = \int_M q(a_1 r_2 + a_2 r_1 + r_1 r_2) \, \dd V
\end{equation}
where $a_1, a_2$ are smooth functions in $M$ independent of $\tau$. We show that the right hand side converges to $0$ as $\tau \to \infty$. To do this, fix $\eps > 0$ and write $q = q^{\sharp} + q^{\flat}$ where $q^{\sharp} \in L^{\infty}(M)$, $\norm{q^{\sharp}}_{L^{n/2}} \leq \norm{q}_{L^{n/2}}$, and $\norm{q^{\flat}}_{L^{n/2}} \leq \eps$. The right hand side of \eqref{q_integral_identity_2} is bounded by 
\begin{multline*}
\left\lvert \int_M q(a_1 r_2 + a_2 r_1 + r_1 r_2) \, \dd V \right\rvert \\
\lesssim \norm{q^{\sharp}}_{L^{\infty}}(\norm{r_1}_{L^2} + \norm{r_2}_{L^2} + \norm{r_1}_{L^2} \norm{r_2}_{L^2}) \\
 + \norm{q^{\flat}}_{L^{n/2}} (\norm{r_1}_{L^{\frac{2n}{n-2}}} + \norm{r_2}_{L^{\frac{2n}{n-2}}} + \norm{r_1}_{L^{\frac{2n}{n-2}}} \norm{r_2}_{L^{\frac{2n}{n-2}}}).
\end{multline*}
Using the bounds for $r_j$, if $\tau$ is sufficiently large then the last quantity is $\lesssim \eps$. This shows that the right hand side of \eqref{q_integral_identity_2} goes to $0$ as $\tau \to \infty$.

Extend $q$ by zero into $T$ and interpret the left hand side of \eqref{q_integral_identity_2} as an integral over $T$. Taking the limit as $\tau \to \infty$, we obtain that 
\begin{equation*}
\int_{-\infty}^{\infty} \int_0^{\infty} \int_{S^{n-2}} q(x_1,r,\theta) \e^{i\lambda(x_1+ir)} b(\theta) \,\dd x_1 \,\dd r \, \dd \theta = 0.
\end{equation*}
This statement is true for all choices of $\omega \in \tilde{M}_0 \setminus M_0$, for all real numbers $\lambda$, and for all functions $b \in C^{\infty}(S^{n-2})$. Since $q \in L^1(M)$, Fubini's theorem shows that $q(\,\cdot\,,r,\theta)$ is in $L^1(\mR)$ for a.e.~$(r,\theta)$. Consequently 
\begin{equation} \label{flambda_integrals}
 \int_{S^{n-2}} \int_0^{\infty} f_{\lambda}(r,\theta) \e^{-\lambda r} b(\theta) \,\dd r \,\dd \theta = 0
\end{equation}
where $f_{\lambda} \in L^1(M_0)$ is the function given by 
\begin{equation*}
f_{\lambda}(r,\theta) = \int_{-\infty}^{\infty} \e^{i\lambda x_1} q(x_1,r,\theta) \,\dd x_1.
\end{equation*}
If $\abs{\lambda}$ is sufficiently small, it follows from Lemma \ref{lemma_uniqueness_integrals} below that the vanishing of the integrals \eqref{flambda_integrals} for all choices $\omega$ and $b$ implies that $f_{\lambda} = 0$. Since $q(\,\cdot\,,r,\theta)$ is a compactly supported function in $L^1(\mR)$ for a.e.~$(r,\theta)$, the Paley-Wiener theorem shows that $q(\,\cdot\,,r,\theta) = 0$ for such $(r,\theta)$. Consequently $q_1 = q_2$.
\end{proof}

\section{Attenuated ray transform} \label{sec_raytransform}

It remains to show the following lemma which was used in the proof of Theorem \ref{theorem_main}.

\begin{lemma} \label{lemma_uniqueness_integrals}
Let $(M_0,g_0)$ be an $(n-1)$-dimensional simple manifold, and let $f \in L^1(M_0)$. Consider the integrals 
\begin{equation*}
 \int_{S^{n-2}} \int_0^{\tau(\omega,\theta)} f(r,\theta) \e^{-\lambda r} b(\theta) \,\dd r \,\dd \theta
\end{equation*}
where $(r,\theta)$ are polar normal coordinates in $(M_0,g_0)$ centered at some $\omega \in \partial M_0$, and $\tau(\omega,\theta)$ is the time when the geodesic $r \mapsto (r,\theta)$ exits $M_0$. If $\abs{\lambda}$ is sufficiently small, and if these integrals vanish for all $\omega \in \partial M_0$ and all $b \in C^{\infty}(S^{n-2})$, then $f = 0$.
\end{lemma}

The last result is related to the vanishing of the attenuated geodesic ray transform of the function $f$ on $M_0$. For the following facts see \cite{DPSU}, \cite{PeU}, \cite{Sh}. To define the ray transform, we consider the unit sphere bundle 
$$
SM_0 = \bigcup_{x \in M_0} S_x, \qquad S_x = \big\{ (x,\xi) \in T_x M_0 \,;\, \abs{\xi} = 1\big\}.
$$
This manifold has boundary $\partial (S M_0) = \{ (x,\xi) \in S M_0 \,;\, x \in \partial M_0 \}$ which is the union of the sets of inward and outward pointing vectors, 
$$
\partial_{\pm} (S M_0) = \big\{ (x,\xi) \in S M_0 \,;\, \pm \langle \xi, \nu \rangle \leq 0 \big\}.
$$
Here $\nu$ is the outer unit normal vector to $\partial M_0$. Note that $\partial_+(S M_0)$ is a manifold whose boundary consists of all the tangential directions $\{(x,\xi) \in \partial(S M_0) \,;\, \langle \xi, \nu \rangle = 0 \}$. Thus the space $C^{\infty}_0((\partial_+(S M_0))^{\text{int}})$ contains all smooth functions on $\partial_+(S M_0)$ vanishing near tangential directions.

Denote by $t \mapsto \gamma(t,x,\xi)$ the unit speed geodesic starting at $x$ in direction $\xi$, and let $\tau(x,\xi)$ be the time when this geodesic exits $M_0$. Since $(M_0,g_0)$ is simple, $\tau(x,\xi)$ is finite for each $(x,\xi) \in S M_0$. We also write $\varphi_t(x,\xi) = (\gamma(t,x,\xi), \dot{\gamma}(t,x,\xi))$ for the geodesic flow.

The geodesic ray transform, with constant attenuation $-\lambda$, acts on functions on $M_0$ by 
$$
T_{\lambda} f(x,\xi) = \int_0^{\tau(x,\xi)} f(\gamma(t,x,\xi)) \e^{-\lambda t} \, \dd t, \qquad (x,\xi) \in \partial_+ (S M_0).
$$
In Lemma \ref{lemma_uniqueness_integrals}, if $f$ were a continuous function, one could choose $b(\theta)$ to approximate a delta function at fixed angles $\theta$ and obtain that 
$$
\int_0^{\tau(\omega,\theta)} f(r,\theta) \e^{-\lambda r} \,\dd r =  0
$$
for any $\omega \in \partial M_0$ and any $\theta \in S^{n-2}$. Since $(r,\theta)$ are polar normal coordinates the curves $r \mapsto (r,\theta)$ are geodesics in $(M_0,g_0)$, and this would imply that 
$$
T_{\lambda} f(x,\xi) = 0 \qquad \text{for all } (x,\xi) \in \partial_+(S M_0).
$$
One has the following injectivity result from \cite[Theorem 7.1]{DKSaU}. (If $M_0$ is two-dimensional the smallness assumption on the attenuation coefficient was recently removed in \cite{SaU}.)

\begin{prop} \label{prop_raytransform_uniqueness}
Let $(M_0,g_0)$ be a simple manifold. There exists $\eps > 0$ such that if $\lambda$ is a real number with $\abs{\lambda} < \eps$ and if $f \in C^{\infty}(M)$, then the condition $T_{\lambda} f(x,\xi) = 0$ for all $(x,\xi) \in \partial_+(S M_0)$ implies that $f = 0$.
\end{prop}

The previous argument together with Proposition \ref{prop_raytransform_uniqueness} proves Lemma \ref{lemma_uniqueness_integrals} for smooth $f$. However, this requires well defined restrictions of $f$ to all geodesics and it is not obvious how to do this when $f \in L^1$. We circumvent this problem by using duality and the ellipticity of the normal operator $T_{\lambda}^* T_{\lambda}$.

We will need a few facts. Below we write
$$
h_{\psi}(x,\xi) = h(\varphi_{-\tau(x,-\xi)}(x,\xi)), \quad (x,\xi) \in S M_0
$$
for $h \in C^{\infty}(\partial_+ (SM_0))$, and 
$$
(h, \tilde{h})_{L^2_{\mu}(\partial_+(S M_0))} = \int_{\partial_+ (S M_0)} h \tilde{h} \mu \,\dd(\partial (S M_0))
$$
where $\mu(x,\xi) = -\langle \xi, \nu(x) \rangle$ and $\dd N$ is the natural Riemannian volume form on a manifold $N$.

\begin{lemma} \label{lemma_santalo_formula}
(Santal\'o formula) If $F: S M_0 \to \R$ is continuous then 
\begin{multline*}
\int_{S M_0} F \,\dd(S M_0) \\ = \int_{\partial_+(S M_0)} \int_0^{\tau(x,\xi)} F(\varphi_t(x,\xi)) \mu(x,\xi) \,\dd t \,\dd(\partial(S M_0))(x,\xi).
\end{multline*}
\end{lemma}
\begin{proof}
See \cite[Lemma A.8]{DPSU}.
\end{proof}

\begin{lemma} \label{lemma_raytransform_adjoint}
If $f \in C^{\infty}(M_0)$ and $h \in C^{\infty}_0((\partial_+ (S M_0))^{\text{int}})$ then 
$$
(T_{\lambda} f, h)_{L^2_{\mu}(\partial_+(S M_0))} = (f, T_{\lambda}^* h)_{L^2(M_0)}
$$
where 
$$
T_{\lambda}^* h(x) = \int_{S_x} \e^{-\lambda \tau(x,-\xi)} h_{\psi}(x,\xi) \, \dd S_x(\xi), \qquad x \in M_0.
$$
\end{lemma}
\begin{proof}
By the Santal\'o formula 
\begin{align*}
 (T_{\lambda} f, h)&_{L^2_{\mu}(\partial_+(S M_0))} \\ &= \int_{\partial_+ (S M_0)} \int_0^{\tau(x,\xi)} \e^{-\lambda t} f(\gamma(t,x,\xi)) h \mu \,\dd t \,\dd(\partial(S M_0)) \\
 &= \int_{\partial_+ (S M_0)} \int_0^{\tau(x,\xi)} \e^{-\lambda t} f(\varphi_t(x,\xi)) h_{\psi}(\varphi_t(x,\xi)) \mu \,\dd t \,\dd(\partial(S M_0)) \\
 &= \int_{S M_0} \e^{-\lambda \tau(x,-\xi)} f(x) h_{\psi}(x,\xi) \,\dd(S M_0) \\
 &= \int_{M_0} f(x) \left( \int_{S_x} \e^{-\lambda \tau(x,-\xi)} h_{\psi}(x,\xi) \,\dd S_x(\xi) \right) \,\dd V(x).
\end{align*}
This proves Lemma \ref{lemma_raytransform_adjoint}.
\end{proof}

\begin{lemma} \label{lemma_normal_elliptic}
$T_{\lambda}^* T_{\lambda}$ is a self-adjoint elliptic pseudodifferential operator of order $-1$ in $M_0^{\text{int}}$.
\end{lemma}
\begin{proof}
This is contained in \cite[Proposition 2]{FSU}, but for completeness we also include a proof in Appendix \ref{appendix_normaloperator}.
\end{proof}

\begin{proof}[Proof of Lemma \ref{lemma_uniqueness_integrals}]
The first step is to extend $(M_0,g_0)$ to a slightly larger simple manifold and to extend $f$ by zero. In this way we can assume that $f$ is compactly supported in $M_0^{\text{int}}$.

We let $b$ also depend on $\omega$ and change notations to write the assumption in the lemma in the form 
$$
\int_{S_x} \int_0^{\tau(x,\xi)} \e^{-\lambda t} f(\gamma(t,x,\xi)) b(x,\xi) \,\dd t \,\dd S_x(\xi) = 0
$$
for all $x \in \partial M_0$ and $b \in C^{\infty}_0((\partial_+(S M_0))^{\text{int}})$. Next we make the choice $b(x,\xi) = h(x,\xi) \mu(x,\xi)$ for $h \in C^{\infty}_0((\partial_+(S M_0))^{\text{int}})$ and integrate the last identity over $\partial M_0$ to obtain 
$$
\int_{\partial_+(S M_0)} \int_0^{\tau(x,\xi)} \e^{-\lambda t} f(\gamma(t,x,\xi)) h(x,\xi) \mu  \,\dd t \,\dd (\partial(S M_0)) = 0.
$$
We are now in the same situation as in the proof of Lemma \ref{lemma_raytransform_adjoint}, and invoking the Santal\'o formula implies 
$$
\int_{M_0} f(x) T_{\lambda}^* h(x) \,\dd V(x) = 0
$$
for all $h \in C^{\infty}_0((\partial_+(S M_0))^{\text{int}})$. Note that the last integral is absolutely convergent because $f \in L^1(M_0)$, and also the previous steps are justified by Fubini's theorem.

It remains to choose $h = T_{\lambda} \varphi$ for $\varphi \in C^{\infty}_0(M_0^{\text{int}})$ to obtain that 
$$
\int_{M_0} f(x) T_{\lambda}^* T_{\lambda} \varphi(x) \,\dd V(x) = 0.
$$
Since $T_{\lambda}^* T_{\lambda}$ is self-adjoint, we have 
$$
\int_{M_0} (T_{\lambda}^* T_{\lambda} f(x)) \varphi(x) \,\dd V(x) = 0.
$$
This is valid for all test functions $\varphi$, so $T_{\lambda}^* T_{\lambda} f = 0$. By ellipticity, since $f$ was compactly supported in $M_0^{\text{int}}$, it follows that $f \in C^{\infty}_0(M_0^{\text{int}})$. One can now use the argument for smooth $f$ given above, together with the injectivity result (Proposition \ref{prop_raytransform_uniqueness}), to conclude the proof that $f = 0$.
\end{proof}

\end{section}
%
%
\appendix
\section{Wellposedness} \label{appendix_wellposedness}

Here we recall the standard arguments that show wellposedness of the Dirichlet problem for $-\Delta_g + q$ on a compact oriented manifold $(M,g)$ with smooth boundary and with $q \in L^{n/2}(M)$, $n \geq 3$. Consider first the inhomogeneous problem for the Schr\"odinger equation,
\begin{equation} \label{inhomogeneous_dirichlet}
(-\Delta_g + q)u = F \text{ in } M, \quad u|_{\partial M} = 0.
\end{equation}
The bilinear form related to this problem is 
\begin{equation*}
B(u,v) = \int_M \big(\langle \dd u, \dd \bar{v} \rangle + qu\bar{v}\big) \, \dd V, \quad u, v \in H^1_0(M),
\end{equation*}
where $\langle \,\cdot\,,\,\cdot\, \rangle$ is the complex-linear inner product of $1$-forms and $\dd V$ is the volume form on $(M,g)$. By the Sobolev embedding $H^1(M) \subseteq L^{\frac{2n}{n-2}}(M)$ and by H\"older's inequality, $B$ is a bounded bilinear form on $H^1_0(M)$. Writing $q = q^{\sharp} + q^{\flat}$ where $q^{\sharp} \in L^{\infty}(M)$ and $\norm{q^{\flat}}_{L^{n/2}(M)}$ is small, we obtain from Poincar\'e's inequality that 
\begin{equation*}
B(u,u) \geq c \norm{u}_{H^1(M)}^2 - C \norm{u}_{L^2(M)}^2, \quad u \in H^1_0(M).
\end{equation*}
This shows that $B + C$ is coercive, and by the Lax-Milgram lemma, compact Sobolev embedding and the Fredholm theorem, the equation \eqref{inhomogeneous_dirichlet} has a unique solution $u \in H^1_0(M)$ for any $F \in H^{-1}(M)$ if one is outside a countable set of eigenvalues.

We can now consider the Dirichlet problem 
\begin{equation} \label{homogeneous_dirichlet2}
(-\Delta_g + q)u = 0 \text{ in } M, \quad u|_{\partial M} = f.
\end{equation}
We assume that $0$ is not a Dirichlet eigenvalue, and it follows from the above discussion that for any $f \in H^{1/2}(\partial M)$ there is a unique solution $u \in H^1(M)$. The DN map is formally defined as the map 
\begin{align*}
\Lambda_{g,q}: H^{1/2}(\partial M) &\to H^{-1/2}(\partial M) \\ f &\mapsto \partial_{\nu} u|_{\partial M}.
\end{align*}
More precisely, if $f \in H^{1/2}(\partial M)$ we define $\Lambda_{g,q} f$ weakly as the function in $H^{-1/2}(\partial M)$ which satisfies 
\begin{equation*}
\int_{\partial M} (\Lambda_{g,q} f) \bar{h} \,\dd S = \int_M \big(\langle \dd u, \dd \bar{v} \rangle + qu \bar{v}\big) \,\dd V
\end{equation*}
where $u$ is the unique solution of \eqref{homogeneous_dirichlet2}, and $v$ is any extension in $H^1(M)$ of $h$ ($v|_{\partial M} = h$). Then $\Lambda_{g,q}$ is a bounded map $H^{1/2}(\partial M) \to H^{-1/2}(\partial M)$ again by H\"older and Sobolev embedding.

The DN map satisfies in the weak sense 
$$
\int_{\partial M} (\Lambda_{g,q} f) \bar{h} \,\dd S = \int_{\partial M} f \overline{\Lambda_{g,\bar{q}} h} \,\dd S.
$$
To see this, let $u, v \in H^1(M)$ solve $(-\Delta_g + q)u = 0$, $u|_{\partial M} = f$ and $(-\Delta_g + \bar{q}) v = 0$, $v|_{\partial M} = h$. Then 
\begin{align*}
 \int_{\partial M} (\Lambda_{g,q} f) \bar{h} \,\dd S &= \int_M \big(\langle \dd u, \dd \bar{v} \rangle + qu \bar{v}\big) \,\dd V 
 \\ &= \overline{\int_M \big(\langle \dd v, \dd \bar{u} \rangle + \bar{q} v \bar{u}\big) \,\dd V} 
 = \int_{\partial M} \overline{(\Lambda_{g,\bar{q}} h)} f \,\dd S.
\end{align*}
As a consequence, we have the basic integral identity used in the uniqueness proof.

\begin{lemma} \label{lemma_integral_identity}
If $q_1, q_2 \in L^{n/2}(M)$ and $\Lambda_{g,q_1} = \Lambda_{g,q_2}$, then 
$$
\int_M (q_1-q_2) u_1 u_2 \,\dd V = 0
$$
for any $u_j \in H^1(M)$ with $(-\Delta_g + q_1) u_1 = 0$ in $M$, $(-\Delta_g + q_2) u_2 = 0$ in $M$.
\end{lemma}
\begin{proof}
Follows from the computation 
\begin{align*}
0 &= \int_{\partial M} (\Lambda_{g,q_1}-\Lambda_{g,q_2}) (u_1|_{\partial M}) u_2 \,\dd S \\
& = \int_{\partial M} \big( \Lambda_{g,q_1}(u_1|_{\partial M}) u_2 - u_1 \overline{\Lambda_{g,\bar{q}_2}(\bar{u}_2|_{\partial M})} \big) \,\dd S
\end{align*}
and the definition of the DN maps.
\end{proof}

\section{Normal operator} \label{appendix_normaloperator}

The setting is the compact simple Riemannian manifold $(M_{0},g_{0})$ of dimension $n-1$. 
Let $T_{\lambda}$ be the attenuated ray transform as in Section \ref{sec_raytransform}. We will prove Lemma \ref{lemma_normal_elliptic}. Write 
$$
\psi(x,\xi) = \varphi_{-\tau(x,-\xi)}(x,\xi).
$$
We compute the normal operator $T_{\lambda}^* T_{\lambda} f$ for $f \in C^{\infty}_0(M_0^{\text{int}})$
\begin{align*}
&T_{\lambda}^* T_{\lambda} f(x) \\ &= \int_{S_x} \e^{-\lambda \tau(x,-\xi)} (T_{\lambda} f)_{\psi}(x,\xi) \,\dd S_x(\xi) \\
 &= \int_{S_x} \e^{-\lambda \tau(x,-\xi)} \int_0^{\tau(\psi(x,\xi))} \e^{-\lambda t} f(\gamma(t,\psi(x,\xi))) \,\dd t \,\dd S_x(\xi) \\
 &= \int_{S_x} \e^{-\lambda \tau(x,-\xi)} \int_0^{\tau(x,-\xi)+\tau(x,\xi)} \e^{-\lambda t} f(\gamma(t,\psi(x,\xi))) \,\dd t \,\dd S_x(\xi) \\
\intertext{and using changes of variables we get for the last integral expression}
 &T_{\lambda}^* T_{\lambda} f(x) \\ &= \int_{S_x} \e^{-\lambda \tau(x,-\xi)} \int_{-\tau(x,-\xi)}^{\tau(x,\xi)} \e^{-\lambda (s+\tau(x,-\xi))} 
 f(\gamma(s,x,\xi)) \,\dd s \,\dd S_x(\xi) \\
&= \int_{S_x} \left[ \int_{-\tau(x,-\xi)}^0 + \int_0^{\tau(x,\xi)} \right] \e^{- \lambda (s+2\tau(x,-\xi))} f(\gamma(s,x,\xi)) \,\dd s \,\dd S_x(\xi) \\
&=  \int_{S_x} \int_0^{\tau(x,\xi)} \left[ \e^{-2 \lambda \tau(x,-\xi)} \e^{-\lambda s} + \e^{-2 \lambda \tau(x,\xi)} \e^{\lambda s} \right] f(\gamma(s,x,\xi)) \,\dd s \,\dd S_x(\xi).
\end{align*}
Changing variables $y = \exp_x(s\xi)$ shows that 
$$
T_{\lambda}^* T_{\lambda} f(x) = \int_{M_0} K_{\lambda}(x,y) f(y) \,\dd V(y)
$$
where 
\begin{align*}
K_{\lambda}(x,y) = \frac{(\e^{-\lambda \phi_+(x,y)} +  \e^{-\lambda \phi_-(x,y)}) }{d_{g_0}^{n-2}(x,y)} 
\bigg(\frac{\det g_0(x)}{\det g_0(y)}\bigg)^{\frac{1}{2}} |\det(\exp_x^{-1})'(x,y)|.
\end{align*}
with
     $$ \phi_{\pm} = 2\tau(x,\mp \mathop{\rm grad}\nolimits_yd_{g_0}(x,y)) \pm d_{g_0}(x,y). $$
The functions $\phi_{\pm}$ are smooth away from the diagonal $x=y$, and their $k$-th order derivatives behave as $d_{g_0}(x,y)^{-k}$.
Note that $\det (\exp^{-1})'$ stands for the Jacobian determinant of 
     \begin{align*}
          \exp^{-1} : M_{0} \times M_{0} &\to \R^{n-1} \\
          (x,y) & \mapsto \exp_{x}^{-1}(y).
      \end{align*}
The kernel of the normal operator is symmetric
          $$ K_{\lambda}(y,x) = K_{\lambda}(x,y) $$
and the singular support of this kernel is the diagonal in $M_{0} \times M_{0}$. 

We will now prove that the operator $P_{\lambda}$ with kernel $K_{\lambda}$ is actually a pseudodifferential operator. 
The first observation in that direction is that in coordinates
\begin{align}
\label{Appendix:distance}
       d^2_{g_{0}}(x,y) = a_{jk}(x,y)(x^{j}-y^{j})(x^{k}-y^{k}) 
\end{align}
with $a^{jk}(x,x)=g_{0}^{jk}(x)$. Indeed the square of the distance vanishes at second order and its Hessian at $x=y$ is twice the metric.
This can be seen from the well known formula
     \begin{align*}
         \nabla^2 \phi (y)(\theta,\theta) = \frac{\d^2}{\d t^2} \phi(\exp_y t \theta)\Big|_{t=0}
     \end{align*}
and the fact that if $|\theta|_{g_0}=1$ then $d^2_{g_0}(\exp_y t\theta,y)=t^2$.
To prove that $P_{\lambda}$ is a pseudodifferential operator in $\Psi^{-1}(M^{\rm int})$ we need to show that for any couple of cutoff functions
$(\psi,\chi)$ supported in charts of $M^{\rm int}$, the operator with kernel
        $$ \tilde{K}_{\lambda}(x,y) = \psi(x) K_{\lambda}(x,y) \sqrt{\det g_0(y)}\chi(y) $$
expressed in coordinates%
\footnote{By a slight abuse of notations, to lighten the exposition, we don't write the pullback by the coordinates 
and think of $x$ and $y$ as variables in $\R^{n-1}$.}, 
is a pseudodifferential operator on $\R^{n-1}$ with symbol in $S^{-1}$. Because of its form and of \eqref{Appendix:distance}, the kernel satisfies
\begin{align}
\label{Appendix:KerBehav}
      |\d_{x}^{\alpha}\d_y^{\beta}\tilde{K}_{\lambda}(x,x-y)| &\leq C_{\alpha} |y|^{-n+2-|\beta|}
\end{align}
and has compact support in $\R^{n-1} \times \R^{n-1}$.

Such operators are pseudodifferential operators and this can easily be seen in the following way:
the symbol associated with such an operator is
\begin{align*}
      \tilde{p}_{\lambda}(x,\xi) &= \int \tilde{K}_{\lambda}(x,x-y)\e^{- i y \cdot \xi} \, \dd y 
\end{align*}
For cutoff functions $\psi$ and $\chi$ whose supports don't intersect, the previous symbol is a Schwartz function because the kernel is a smooth 
compactly supported function. So we are only interested in those symbols corresponding to kernels $\tilde{K}_{\lambda}(x,x-y)$ which are 
supported close to $\R^{n-1} \times\{0\}$. 
In that case, we use a dyadic partition of unity
           $$ 1 = \sum_{\mu=-\infty}^{\infty} \chi(2^{-\mu}z), $$
with $\chi$ a function supported in an annulus, to decompose the symbol as a sum of terms of the form
       \begin{align*}
          2^{\mu (n-1)}\int \e^{i 2^{\mu} y\cdot \xi}  \chi(y) \tilde{K}_{\lambda}(x,x-2^{\mu}y) \, \dd y.
     \end{align*}
Note that because of the compact support of the kernel, these terms vanish when $\mu$ is large, so we are mainly concerned with the 
terms where $\mu$ is less than some positive integer, say $N$. Because of the behaviour \eqref{Appendix:KerBehav}, the rescaled kernel
$\tilde{K}_{\lambda}(x,x-2^{\mu}y)$ is uniformly bounded by $2^{-\mu(n-2)}$ as well as all its derivatives.
Applying the non-stationary phase when $|\xi|\geq 1$ and  $2^{\mu}\xi$ is large we get 
     \begin{align*}
          |\tilde{p}_{\lambda}(x,\xi)| &\lesssim   \sum_{\mu \leq N, \, 2^{\mu}|\xi| \geq 1} 2^{\mu} (2^{\mu}|\xi|)^{-N}
          + \sum_{\mu \leq N, \, 2^{\mu}|\xi| \leq 1}  2^{\mu} \\ &\lesssim (1+|\xi|)^{-1}
     \end{align*}
Repeating this argument for the derivatives of this function, we get that $\tilde{p}_{\lambda}$ is a classical symbol of order $-1$.
          
Let us concentrate on $p_0$: we have
     $$ \tilde{K_0}(x,y) =  \psi(x)  \frac{\det(\exp_x^{-1})'(x,y)}{d_{g_0}^{n-2}(x,y)}  \sqrt{\det g_0(y)} \chi(y)   $$
from the previous computation, we see that taking $x=y$ in the nonsingular factors, yields error terms whose kernel are less singular
by an order of $|x-y|$, i.e. errors with symbols of order $(1+|\xi|)^{-2}$. Therefore in terms of the principal symbol, it suffices to compute
\begin{multline*}
      \psi(x) \chi(x) \times \sqrt{\det g_0(x)} \int  \frac{\e^{-i y\cdot \xi} }{|g_0(x)y\cdot y|^{\frac{n-2}{2}}} \, \dd y \\
      = c_n |g_0^{-1}(x)\xi \cdot \xi|^{-\frac{1}{2}} \psi(x) \chi(x).
\end{multline*}
Finally, these observations show that $P_{0}$ has as principal symbol a multiple of 
    $$ |\xi|_{g_{0}}^{-1}=\frac{1}{\sqrt{g_0^{jk}(x)\xi_{j} \xi_{k}}} $$
and since the principal symbol of $P_{\lambda}$ depends smoothly on $\lambda$, it doesn't vanish for $\lambda$ small enough.
This means that for $\lambda$ small enough, $P_{\lambda}$ is an elliptic self-adjoint pseudodifferential operator of order $-1$.
%

%
%

%
%

\begin{thebibliography}{10}


\bibitem{calderon}
A.~P. Calder{\'o}n, \emph{On an inverse boundary value problem}, Seminar on
  Numerical Analysis and its Applications to Continuum Physics, Soc. Brasileira
  de Matem{\'a}tica, R{\'i}o de Janeiro, 1980.
  
\bibitem{Cha}
S.~Chanillo, \emph{{A problem in electrical prospection and a $n$-dimensional Borg-Levinson theorem}}, Proc. Amer. Math. Soc. \textbf{108} (1990), 761--767.

\bibitem{DPSU}
N.~S.~Dairbekov, G.~P.~Paternain, P.~Stefanov, and G.~Uhlmann, \emph{{The boundary rigidity problem in the presence of a magnetic field}}, Adv. Math. \textbf{216} (2007), 535--609.

\bibitem{DKSaU}
D.~Dos Santos~Ferreira, C.~E. Kenig, M.~Salo, and G.~Uhlmann,
  \emph{{Limiting Carleman weights and anisotropic inverse problems}}, Invent. Math. \textbf{178} (2009), 119--171.

\bibitem{dksu}
D.~Dos Santos~Ferreira, C.~E. Kenig, J.~Sj{\"o}strand, and G.~Uhlmann,
  \emph{{Determining a magnetic Schr{\"o}dinger operator from partial Cauchy
  data}}, Comm. Math. Phys. \textbf{271} (2007), 467--488.
  
\bibitem{FSU}
B.~Frigyik, P.~Stefanov, and G.~Uhlmann, \emph{{The X-ray transform for a generic family of curves and weights}}, J. Geom. Anal. \textbf{18} (2008), 89--108.

\bibitem{GS} C.~Guillarmou, A.~Sa Barreto, \textit{Inverse problems for Einstein manifolds}, Inverse Probl. Imaging \textbf{3} (2009), 1--15.


\bibitem{hormander}
L.~H\"ormander, \emph{The analysis of linear partial differential operators}, Springer Verlag, 1994.

\bibitem{Jerison}
D.~Jerison, \emph{{Carleman inequalities for the Dirac and Laplace operators and unique continuation}}, Adv. Math. \textbf{63} (1986), 118--134.

\bibitem{JK}
D.~Jerison and C.~E.~Kenig, \emph{{Unique continuation and absence of positive eigenvalues for Schr\"odinger operators}}, Ann. of Math. \textbf{121} (1985), 463--494.

\bibitem{KKL} A.~Katchalov, Y.~Kurylev, M.~Lassas, \emph{{Inverse boundary spectral problems}}, Monographs and Surveys in Pure and Applied Mathematics 123, 
Chapman Hall/CRC-press, 2001.

\bibitem{KRS}
C.~E.~Kenig, A.~Ruiz, and C.~D.~Sogge, \emph{{Uniform Sobolev inequalities and unique continuation for second order constant coefficient differential operators}}, Duke Math. J. \textbf{55} (1987), 329--347.

\bibitem{KSaU}
C.~E.~Kenig, M.~Salo, and G.~Uhlmann, \emph{{Inverse problems for the anisotropic Maxwell equations}},Duke Math. J. \textbf{157} (2011), 369--419.


\bibitem{ksu}
C.~E. Kenig, J.~Sj{\"o}strand, and G.~Uhlmann, \emph{{The Calder{\'o}n problem with partial data}}, Ann. of Math. \textbf{165} (2007), 567--591.

\bibitem{KT_nonsmooth}
H.~Koch and D.~Tataru, \emph{Carleman estimates and unique continuation for
  second order elliptic equations with nonsmooth coefficients}, Comm. Pure
  Appl. Math. \textbf{54} (2001), no.~3, 339--360.

\bibitem{KT}
H.~Koch and D.~Tataru, \emph{{Dispersive estimates for principally normal pseudodifferential operators}}, Comm. Pure Appl. Math. \textbf{58} (2005), 217--284.



\bibitem{LN}
R.~Lavine and A.~Nachman, unpublished (personal communication), announced in A. I. Nachman, \textit{Inverse scattering
at fixed energy}, Proceedings of the Xth Congress on Mathematical Physics, L. SchmŸdgen (Ed.), Leipzig, Germany, 1991, 434--441, Springer-Verlag.

\bibitem{LeU} J.~Lee, G.~Uhlmann, \textit{Determining anisotropic real-analytic conductivities by boundary measurements}, Comm. Pure Appl. Math.,  \textbf{42} (1989), 1097--1112.

\bibitem{nachman_reconstruction}
A.~Nachman, \emph{Reconstructions from boundary measurements}, Ann. of Math. \textbf{128} (1988), 531--576.


\bibitem{PeU} 
L.~Pestov, G.~Uhlmann, \emph{Two dimensional compact simple Riemannian manifolds are boundary distance rigid}, Ann. of Math. \textbf{161} (2005), 1089--1106. 

\bibitem{SaU} 
M.~Salo, G.~Uhlmann, \emph{The attenuated ray transform on simple surfaces}, J. Diff. Geom. (to appear), arXiv:1004.2323.

\bibitem{Sh}
V.~A.~Sharafutdinov, \emph{{Integral geometry of tensor fields}},  in Inverse and Ill-Posed Problems Series, VSP 1994.

\bibitem{Shen}
Z.~Shen, \emph{{On absolute continuity of the periodic Schr\"odinger operators}}, IMRN \textbf{1} (2001), 1--31.

\bibitem{Sogge_cluster}
C.~D.~Sogge, \emph{{Concerning the $L^p$ norm of spectral clusters for second-order elliptic operators on compact manifolds}}, 
J. Funct. Anal. \textbf{77} (1988), 123--138.

\bibitem{Sogge_psdo}
C.~D.~Sogge, \emph{{Strong uniqueness theorems for second order elliptic
differential equations}}, Amer. J. Math. \textbf{112} (1990), 943--984.

\bibitem{Sogge}
C.~D.~Sogge, \emph{{Fourier integrals in classical analysis}}, Cambridge University Press, 1993.

\bibitem{Stein}
E.~M.~Stein, \emph{{Harmonic Analysis: real variable methods, orthogonality and oscillatory integrals}}, Princeton University Press, 1993.

%
\bibitem{Sy}
J.~Sylvester, \emph{{An anisotropic inverse boundary value problem}}, Comm. Pure Appl. Math. \textbf{43} (1990), 201--232.

\bibitem{sylvesteruhlmann}
J.~Sylvester and G.~Uhlmann, \emph{A global uniqueness theorem for an inverse
  boundary value problem}, Ann. of Math. \textbf{125} (1987), 153--169.

\bibitem{U_IP} G.~Uhlmann, \textit{Electrical impedance tomography and Calder{\'o}n's problem}, Inverse Problems \textbf{25} (2009), 123011.

\end{thebibliography}
\end{document}